\documentclass[a4paper,11pt]{amsart}
\usepackage[utf8]{inputenc}

\usepackage{verbatim}
\usepackage{lmodern}
\usepackage{amsfonts}
\usepackage{a4wide}
\usepackage{amsthm}
\usepackage{color}
\usepackage{graphicx}
\usepackage{amsmath,amssymb}
\usepackage{lipsum}
\usepackage{fancyhdr}
\usepackage{hyperref}
\usepackage{doi}
\usepackage{ulem}
\usepackage{enumitem} 
\setlist{itemsep=0.4em}
\definecolor{wineRed}{rgb}{0.7,0,0.3}

\usepackage[numbers]{natbib}

\newcommand{\Lb}{\mathcal L}
\newcommand{\T}{\mathbb T}
\newcommand{\N}{\mathbb N}

\newcommand{\W}{\mathcal W}
\newcommand{\bu}{u}
\newcommand{\be}{e}
\newcommand{\bff}{f}
\newcommand{\bg}{g}
\newcommand{\bw}{w}
\newcommand{\dd}{\,\mathrm d}
\newcommand{\norm}{\nu}
\renewcommand{\div}{\mathrm{div}\,}

\newcommand{\ess}{\mathrm{ess}\,}
\newcommand{\supp}{\mathrm{supp}\,}
\newcommand{\diam}{\mathrm{diam}\,}

\newcommand{\eps}{\varepsilon}

\newcommand{\R}{\mathbb R}

\newcommand{\dist}{\mathrm{dist}}

\newcommand{\Rn}{\mathbb R^n}
\newcommand{\RN}{\mathbb R^N}
\newcommand{\Rm}{\mathbb R^m}
\newcommand{\Rmn}{\mathbb R^{m \times n}}
\newcommand{\Rnm}{\mathbb R^{n \times m}}

\DeclareMathOperator*{\tr}{Tr}

\newtheorem{prop}{Proposition}
\newtheorem{thm}[prop]{Theorem}
\newtheorem{lemma}[prop]{Lemma}

\newtheorem{cor}[prop]{Corollary}
\theoremstyle{remark}
\newtheorem{remark}{Remark}

\def\onedot{$\mathsurround0pt\ldotp$}
\def\cddot{
  \mathbin{\vcenter{\baselineskip.67ex
    \hbox{\onedot}\hbox{\onedot}}%
  }}%
\def\cdddot{
  \mathbin{\vcenter{\baselineskip.67ex
    \hbox{\onedot}\hbox{\onedot}\hbox{\onedot}%
  }}%
}

\newcommand{\ignore}[1]{{}}

\date{\today}

\begin{document}

	\title[Existence of $W^{1,1}$ solutions]{Existence of $W^{1,1}$ solutions to a class of variational problems with linear growth on convex domains}                                 
	\author{Micha{\l} {\L}asica}                                
	\address{Institute of Mathematics of the Polish Academy of Sciences \newline		
		ul.\,\'Sniadeckich 8, 00-656 Warszawa, Poland, ORCID 0000-0002-8365-0484}                                    
               
	\email{mlasica@impan.pl}  
	\author{Piotr Rybka} 
	\address{Institute of Applied Mathematics and Mechanics, University of Warsaw\newline ul.\,Banacha 2, 02-097 Warszawa, Poland, ORCID 0000-0002-0694-8201}                                    
                                 
    \email{rybka@mimuw.edu.pl}
    \thanks{The research of both authors was partially supported by the National Science Center, Poland, through the grant number 2017/26/M/ST1/00700. }
      	\date{\today} 
 \begin{abstract}
	We consider a class of convex integral functionals composed of a term of linear growth in the gradient of the argument, and a fidelity term involving $L^2$ distance from a datum. Such functionals are known to attain their infima in the $BV$ space. Under the assumption that the domain of integration is convex, we prove that if the datum is in $W^{1,1}$, then the functional has a minimizer in $W^{1,1}$. In fact, the minimizer inherits $W^{1,p}$ regularity from the datum for any $p \in [1, +\infty]$. We also obtain a quantitative bound on the singular part of the gradient of the minimizer in the case that the datum is in $BV$. We infer analogous results for the gradient flow of the underlying functional of linear growth. We admit any convex integrand of linear growth.
\end{abstract}
	
 \maketitle

\section{Introduction} 
\noindent We say that a function $\Psi \colon \R^N \to [0, +\infty[$, $N\in \N$ is of \emph{linear growth} (at infinity), if there exist constants $C_1, C_2 >0$ such that
\begin{equation} \label{linear} 
C_1|\xi| \leq \Psi (\xi) \leq C_2 (1+|\xi|) \quad \text{for } \xi \in \R^N. 
\end{equation} 
If we only know that the second inequality in \eqref{linear} is satisfied, we say that $\Psi$ is of \emph{at most linear growth}. 

Let $m \in \mathbb N$ and let $\Omega$ be a bounded domain in $\Rm$. We will write
\[\W = L^2(\Omega) \cap W^{1,1}(\Omega).\]
As an intersection of Banach spaces, $\W$ comes with natural notions of strong and weak convergence. Namely, a sequence, or a generalized sequence, converges (weakly converges) in $\W$ if and only if it converges (weakly converges) in $L^2(\Omega)$ and $W^{1,1}(\Omega)$. Equivalently, a sequence, or a generalized sequence, converges (weakly converges) in $\W$ if and only if it converges (weakly converges) in $W^{1,1}(\Omega)$ and is bounded in $L^2(\Omega)$.
 
Further, let $\Phi \colon \R^m \to [0, +\infty[$ be a convex function of linear growth. Given $\lambda >0$, $\bff \in L^2(\Omega)$, we consider the minimization problem for the functional $E^\lambda_{\bff} \colon \W \to [0, +\infty[$ given by 
\begin{equation}\label{defE} 
E^\lambda_{\bff}(\bw) =  \lambda \int_\Omega \Phi(\nabla \bw) + \frac{1}{2} \int_\Omega |\bw - \bff|^2 .
\end{equation} 
The functional $E^\lambda_{\bff}$ is weakly lower semicontinuous on $\W$. However, this space is not reflexive. Hence, without additional assumptions $E^\lambda_{\bff}$ may fail to attain its infimum. In order to resolve this issue, one may opt to consider instead its lower semicontinuous envelope $\overline E^\lambda_{\bff}$ in $L^2(\Omega)$. This relaxation amounts to extending the effective domain of $E^\lambda_{\bff}$ to $BV(\Omega) \cap L^2(\Omega)$ by the formula 
\[\overline E^\lambda_{\bff}(\bw) = \lambda \int_\Omega \Phi(\nabla^{ac} \bw) + \lambda \int_\Omega \Phi^\infty\left(\tfrac{\nabla^s \bw}{|\nabla^s \bw|}\right) \dd |\nabla^s \bw| + \frac{1}{2} \int_\Omega |\bw - \bff|^2, \] 
where $\nabla \bw = \nabla^{ac} \bw \Lb^m + \nabla^s \bw$, $\nabla^{ac} \bw = \frac{\nabla \bw}{\Lb^m}$ ($\frac{\nabla^s \bw}{|\nabla^s \bw|}$ and $\frac{\nabla \bw}{\Lb^m}$ are Radon-Nikodym derivatives) and 
\[\Phi^\infty \colon \mathbb S^{m - 1} \to [0, + \infty[, \quad \Phi^\infty(\xi) = \lim_{t \to + \infty} \tfrac{\Phi(t \xi)}{t}\]
is the recession function of $\Phi$ \cite{GoffmanSerrin, DemengelTemam}, see also \cite[Theorem 5.47]{AmbrosioFuscoPallara}. The direct method of the calculus of variations produces a minimizer $\bu$ of $\overline E^\lambda_{\bff}$ which by strict convexity is unique.

A~question arises then, to what extent can one control the singularity of measure $\nabla \bu$ in terms of $\nabla \bff$. In particular, what are the conditions implying that the minimizer $\bu$ of $\overline E^\lambda_{\bff}$ belongs to $W^{1,1}(\Omega)$, i.\,e.\;$\bu$ is also a minimizer of $E^\lambda_{\bff}$. Let us mention a few  known results in this direction. In \cite{bcno} and \cite{BonforteFigalli}, it has been established for $m=1$ and $\Phi = |\cdot|$ that $|\nabla \bu|\leq |\nabla \bff|$ in the sense of measures. This was later generalized to the vectorial case (where $\bu, \bff \colon \Omega \to \R^n, n>1$) in \cite{giacomellilasica}. Such an estimate is known to fail if $m>1$. However, analogous estimate was proved for the jump part of measure $|\nabla \bu|$ in \cite{CasellesChambolleNovaga2007, jalalzaijump}. A similar result was obtained for a more general class of integrands $\Phi$ in \cite{Valkonen2015}. Whether an estimate of this kind holds for the Cantor part of measure $|\nabla \bu|$ in $m>1$ remains, to our knowledge, an open question. In \cite{mercier}, it is assumed that $\Omega$ is convex and $\Phi$ is of form $\widetilde \Phi \circ \phi$, where $\phi$ is a norm on $\Rm$ and $\widetilde \Phi$ is of linear growth. Under this condition, it is proved that if $\bff$ admits any modulus of continuity with respect to the dual norm $\phi^*$, then it is inherited by $\bu$. In particular, if $\bff \in W^{1,\infty}(\Omega)$, then $\bu \in W^{1,\infty}(\Omega) \subset W^{1,1}(\Omega)$. On the other hand, in \cite{NakayasuRybka}, the case $m=1$ is considered (with $\Omega = \T$). In this setting it is proved for any convex $\Phi$ with linear growth that if $\bff \in W^{1,1}(\Omega)$, then $\bu \in W^{1,1}(\Omega)$ as well. Here, we generalize this statement to an arbitrary value of $m$.

\begin{thm} \label{w11thm}
	Suppose that $\Omega$ is convex. If $\bff \in W^{1,1}(\Omega)$, then there exists a minimizer $\bu \in W^{1,1}(\Omega)$ of $E^\lambda_{\bff}$. Moreover, for any even, convex function $\widetilde \Psi \colon \R \to [0, +\infty[$ there holds 
	\begin{equation} \label{Psiest}
	\int_\Omega\widetilde \Psi(\Phi(\nabla \bu)) \leq \int_\Omega\widetilde \Psi(\Phi(\nabla \bff)). 
	\end{equation}
\end{thm} 
\noindent 
Note that we never evaluate $\widetilde \Psi$ on negative arguments. We could equivalently assume that $\widetilde \Psi$ is a continuous, convex, non-decreasing function $[0, + \infty[ \to\!\! [0, + \infty[$. Note also that the r.\,h.\,s.\;of \eqref{Psiest} may be infinite.  

As an immediate consequence of Theorem \ref{w11thm}, we deduce that if $\bff \in W^{1,p}(\Omega)$, then $\bu \in W^{1,p}(\Omega)$ with 
\[ \| \Phi(\nabla \bu)\|_{L^p(\Omega)} \leq \| \Phi(\nabla \bff)\|_{L^p(\Omega)} \] 
for $p \in ]1, \infty[$, and therefore also for $p = \infty$. 

Our strategy in the proof of Theorem \ref{w11thm} is first to obtain a version of \eqref{Psiest} for a family of smooth, uniformly convex approximations to $\Phi$. This is done using an energy method. An important point here is that the minimizers of approximations to $E^\lambda_{\bff}$ have $W^{2,2}$ regularity, which is enough to differentiate the Euler-Lagrange system and test it with a suitable function. Estimate \eqref{Psiest} is then used to obtain compactness of approximate minimizers in weak $W^{1,1}$ topology and exhibit a minimizer of $E^\lambda_{\bff}$ as their limit point. 

Since we are unable to localize \eqref{Psiest}, we need to work up to the boundary. For this reason we need convexity of $\Omega$, as it implies that the boundary term that appears in our energy estimate has a definite sign.

In fact, we can also obtain the following quantitative bound on the singular part of the minimizer of $\overline E^\lambda_{\bff}$ in the case that $\bff \in BV(\Omega)$. 
\begin{thm}\label{Dsuthm}
Suppose that $\Omega$ is convex and $\bff \in BV(\Omega)$. Let $\bu$ be the minimizer of $\overline E^\lambda_{\bff}$. We have 
\begin{equation} \label{Dsuest} 
\int_\Omega \Phi^\infty\left(\tfrac{\nabla^s \bu}{|\nabla^s \bu|}\right) \dd |\nabla^s \bu| \leq \int_\Omega \Phi^\infty\left(\tfrac{\nabla^s \bff}{|\nabla^s \bff|}\right) \dd |\nabla^s \bff|. 
\end{equation} 
\end{thm}
\noindent We note that the first (existential) assertion of Theorem \ref{w11thm} follows from Theorem \ref{Dsuthm}. We decided to present the two results as separate theorems because their proofs are somewhat different (although both are based on Lemma \ref{mainlem}). In particular in Theorem \ref{w11thm} the minimizer $\bu$ is exhibited as a limit of a weakly convergent sequence in $W^{1,1}(\Omega)$, without introducing $\overline{E}^\lambda_{\bff}$ and resorting to any weak-$*$ lower semicontinuity result.  

The assumption of convexity of $\Omega$ in Theorems \ref{w11thm} and \ref{Dsuthm} cannot be dropped. In fact, in the case of non-convex $\Omega$, the minimizer of $\overline E^\lambda_{\bff}$ might not belong to $W^{1,1}_{loc}$ even if $\bff$ is smooth up to the boundary, see e.\,g.\;\cite[Example 3]{tetris}. 

During the preparation of this manuscript, we learned about work \cite{Porretta2019}, where the case $\Phi = |\cdot|$ is considered. The author obtains inheritance of $W^{1,\infty}$ regularity without assuming convexity of $\Omega$. Additionally, assuming convexity of $\Omega$, inheritance of $W^{1,p}$ regularity is obtained for $p \in [2, +\infty]$, which is a special case of Theorem \ref{w11thm}.    

We would also like to mention a paper \cite{BeckBulicekGmeineder}, where existence of $W^{1,1}$ solutions is obtained in vectorial setting for functionals of linear growth with a regular enough source term instead of fidelity term. There, $\Phi$ is of form $\widetilde \Phi \circ |\cdot|$, with $\widetilde \Phi$ strictly convex and sufficiently regular with a bound on the tail of $\widetilde \Phi''$. However, $\Omega$ is only assumed to be simply connected.

On a side note, we point out that there are several results concerning solvability in Sobolev spaces of the minimization problem for integral functionals of linear growth with prescribed boundary condition under certain assumptions. For instance, in \cite{MaricondaTreu} suitable restrictions are imposed on the boundary datum, while in \cite{Bildhauer2002} (see also \cite{BeckSchmidt, BeckBulicekMaringova}) a quantitative strict convexity condition is imposed on the integrand. There are also related works on solvability of the least gradient problem in $BV$ with boundary condition prescribed in the trace sense (as opposed to the relaxed sense) and inheritance of (H\" older) continuity from the boundary datum, where various notions of strict convexity of $\Omega$ are assumed, see e.\,g.\;\cite{SternbergWilliamsZiemer, Gorny2018}. In all papers mentioned in this paragraph, except \cite{Bildhauer2002, BeckSchmidt}, only the scalar case is considered. 

A reader may ask whether we can apply our approach to the vectorial case, where $u, f \colon \Omega \to \Rn$, $n>1$. We do not directly use linear order of $\R$ via the comparison principle or such, instead relying only on energy-type estimates. However, we are only able to derive our estimates in the scalar case. In the course of proof of Theorem \ref{w11thm}, see Remark \ref{uwaga}, we explain where our method breaks down for $n>1$. We note that the $W^{2,2}$ regularity result that we show for the uniformly convex approximation is valid also in the vectorial case. We include it in full generality for possible future reference. 

Now, let us define $\overline{F} \colon L^2(\Omega) \to [0, +\infty]$ by the following formula: 
\begin{equation}\label{overlineF} \overline{F}(\bw) = \left\{\begin{array}{ll} \int_\Omega \Phi(\nabla^{ac} \bw) + \int_\Omega \Phi^\infty\left(\tfrac{\nabla^s \bw}{|\nabla^s \bw|}\right) \dd |\nabla^s \bw|& \text{if } \bw \in BV(\Omega),\\ 
+ \infty & \text{otherwise.} \end{array} \right. 
\end{equation} 
The minimization problem for $\overline{E}^\lambda_{\bff}$ coincides with the resolvent problem for the gradient flow of $\overline{F}$. Since $\overline{F}$ is convex and lower semicontinuous, it generates a gradient flow \cite[Corollary 20]{Brezis1971}, i.\,e.\;given $\bu_0 \in D(\overline F) = L^2(\Omega) \cap BV(\Omega)$ there exists exactly one $\bu \in W^{1,2}(0, \infty; L^2(\Omega))$ such that $\bu(0) = \bu_0$ and for a.\,e.\;$t>0$, 
\begin{equation}\label{gradfloweq} 
\bu_t \in - \partial \overline{F}(\bu).
\end{equation}  
As a corollary of our previous results, we obtain \nopagebreak[9]
\begin{thm} \label{gradflow} 
Given $\bu_0 \in L^2(\Omega) \cap BV(\Omega)$, let $\bu \in W^{1,2}(0, \infty; L^2(\Omega))$ be the solution to \eqref{gradfloweq} with $\bu(0)=\bu_0$. For a.\,e.\;$t>0$ there holds 
\begin{equation} \label{Dsuestgf} 
\int_\Omega \Phi^\infty\left(\tfrac{\nabla^s \bu(t)}{|\nabla^s \bu(t)|}\right) \dd |\nabla^s \bu(t)| \leq \int_\Omega \Phi^\infty\left(\tfrac{\nabla^s \bu_0}{|\nabla^s \bu_0|}\right) \dd |\nabla^s \bu_0|. 
\end{equation} 
If moreover $\bu_0 \in W^{1,1}(\Omega)$ then, for a.\,e.\;$t>0$, $\bu(t) \in W^{1,1}(\Omega)$ and for any even, convex function $\widetilde \Psi \colon \R \to [0, +\infty[$ and a.\,e.\;$t>0$,
	\begin{equation} \label{Psiestgf}
	\int_\Omega\widetilde \Psi(\Phi(\nabla \bu(t))) \leq \int_\Omega\widetilde \Psi(\Phi(\nabla \bu_0)).  
	\end{equation}
\end{thm} 

On many occasions, we use a standard approximate identity $(\varphi_\delta)_{\delta>0}$ on $\R^N$, $N \in \N$. This is a family of functions of form $\varphi_\delta = \frac{1}{\delta^N}\varphi\left(\frac{\cdot}{\delta}\right)$, where $\varphi \in C^\infty_c(\R^N, [0,1])$ is a radially symmetric function whose support is contained in the unit ball $B_1(0)$, such that $\int_{\R^N} \varphi=1$. 

Throughout the paper, we use the summation convention except when explicitly stated. Alternatively, we also use index free notation with stacked vertical dots $\cdot$, $\cddot$, $\cdddot$, depending on how many pairs of indices are contracted. A single dot is often omitted, in line with standard notation for multiplying matrices. The symbol $\nabla$ is used to denote derivation with respect to the spatial variable $x \in \Omega$, while $D$ denotes derivatives of functions such as $\Phi$ with respect to Euclidean spaces they are defined on. The notation $|\cdot|$ will stand for the Euclidean norm on $\Rm$, $\R^{m^2}$ etc. 

\section{Convex functions of at most linear growth} 
It is well known that a convex function $\Psi \colon \R^N\to [0, +\infty[$, $N\in \N$ is locally Lipschitz, and hence differentiable $\Lb^N$-a.\,e. This a.\,e.\;defined derivative, which we denote $D\Psi$, belongs to $L^{\infty}_{loc}(\R^N, \R^N)$ and coincides with the distributional derivative of $\Psi$. Furthermore, $D \Psi \in BV_{loc}(\R^N, \R^N)$ \cite{evansgariepy}. In the case  that $\Psi$ is of at most linear growth, the situation is remarkably more convenient. 
\begin{prop} \label{convexprop} 
	 Suppose that $\Psi \colon \R^N\to [0, +\infty[$ is a convex function of at most linear growth. Then 
	 \[D \Psi \in L^\infty(\R^N, \R^N), \quad D^2 (\varphi * \Psi) \in L^\infty(\R^N, \R^{N \times N}),\]
	 for any $\varphi \in C_c(\R^N)$. 
\end{prop} 
\begin{proof} 
	For $\xi_0 \in \R^N$, $i =1, \ldots, N$, let $L_{\xi_0}^i = \{\xi_0 + t  e_i\colon t \in \R\}$ be a line parallel to the $i$-th coordinate axis of $\R^N$. The restriction $\left.\Psi\right|_{L_{\xi_0}^i}\!\!\!\colon \R \to [0, +\infty[$ is convex. Hence, $\left.\frac{\partial \Psi}{\partial p^i}\right|_{L_{\xi_0}^i}\!\!\!$ is monotone. Therefore 
	\begin{equation} \label{linftyboundderivative} 
	- C_2\leq \left.\frac{\partial \Psi}{\partial p^i}\right|_{L_{\xi_0}^i} \!\!\!\leq C_2, 
	\end{equation} 
	lest the second inequality in \eqref{linear} be violated. Since $i$ and $\xi_0$ are arbitrary, we have demonstrated the first part of the assertion. 
	
	By \eqref{linftyboundderivative} and, again, monotonicity of $\left.\frac{\partial \Psi}{\partial p^i}\right|_{L_{\xi_0}^i}$\!\!\!, we have for any $i$, $\xi_0$ the following estimate,
	\[ \left|\frac{\partial^2 \Psi}{(\partial p^i)^2}\right|\left(L_{\xi_0}^i\right) =  \frac{\partial^2 \Psi}{(\partial p^i)^2}\left(L_{\xi_0}^i\right) \leq 2 C_2.\]
	Hence, by Tonelli's theorem, for $\xi \in \R^N$ we obtain the estimate,  
	\[\frac{\partial^2 (\varphi*\Psi)}{(\partial p^i)^2} (\xi)=  \varphi * \frac{\partial^2 \Psi}{(\partial p^i)^2}(\xi) = \int_{\R^N} \varphi(\xi - \zeta) \dd \frac{\partial^2 \Psi}{(\partial p^i)^2}(\zeta) \leq 2C_2 (\diam \supp \varphi)^{N-1} \sup \varphi. \]
	Finally, we note that $\varphi *\Psi$ is convex, and therefore $\left.\varphi*\Psi\right|_{L_{\xi}^i\times L_{\xi}^j}$ for $\xi \in \R^N$, $i,j=1,\ldots, N$ are convex as well. Hence, by Sylvester's criterion, we obtain a bound on mixed derivatives: 
	\[\frac{\partial^2 (\varphi*\Psi)}{\partial p^i\partial p^j} \leq \left(\frac{\partial^2 (\varphi*\Psi)}{(\partial p^i)^2}	\frac{\partial^2 (\varphi*\Psi)}{(\partial p^j)^2}\right)^\frac{1}{2} \leq \frac{1}{2}\left(\frac{\partial^2 (\varphi*\Psi)}{(\partial p^i)^2}	+ \frac{\partial^2 (\varphi*\Psi)}{(\partial p^j)^2}\right),\]
	which completes the proof. 
\end{proof} 

\section{The approximate problem} 
In this section, we introduce a smoothed version of the functional $E^\lambda_{\bff}$. We consider a smooth, uniformly convex approximation $(\Phi_\eps)_{\eps>0}$ of $\Phi$ given by
\begin{equation} \label{Phieps} 
\Phi_\eps(\xi) = (\varphi_\eps * \Phi) (\xi) + \tfrac{\eps}{2}|\xi|^2
\end{equation} 
for $\xi \in \Rm$, where $(\varphi_\eps)_{\eps>0}$ is a standard approximate identity on $\Rm$. Further, we let $(\Omega^\eps)_{\eps>0}$ be a family of smooth, convex subsets of $\Rm$, such that $\Omega \subset \Omega^\eps$ for $\eps >0$ and $\Omega^\eps \to \Omega$ as $\eps \to 0^+$ in Hausdorff distance. We can produce such a family similarly as in \cite[Lemma A.3]{GLMreg}. Given $\bg \in L^2(\Omega^\eps)$, we define $E^{\lambda, \eps}_{\bg} \colon W^{1,2}(\Omega^\eps) \to [0, +\infty[$ by 
\begin{equation}\label{defEeps} 
E^{\lambda, \eps}_{\bg}(\bw) = \lambda \int_{\Omega^\eps} \Phi_\eps(\nabla \bw) + \frac{1}{2} \int_{\Omega^\eps} |\bw - \bg|^2.  
\end{equation}  

\begin{prop} \label{classic} There exists a unique minimizer $\bu^\eps \in W^{1,2}(\Omega^\eps)$ of $E^{\lambda, \eps}_{\bg}$. Furthermore, 
	\begin{itemize}
		\item[(a)] $\bu^\eps \in W^{2,2}(\Omega^\eps)$;
		\item[(b)] $D \Phi_\eps (\nabla \bu^\eps) \in W^{1,2}(\Omega^\eps, \Rm)$ and $D^2 \Phi_\eps (\nabla \bu^\eps) \in L^\infty(\Omega^\eps, \R^{m^2})$; 
		\item[(c)] $\bu^\eps$ satisfies the Euler-Lagrange system 
	\begin{equation} \label{ELeqeps} 
	\bu^\eps - \bg = \lambda\, \div(D \Phi_\eps(\nabla \bu^\eps)) \quad \text{in } \Omega^\eps, 
	\end{equation} 
	\begin{equation} \label{ELbceps} 
	D \Phi_\eps(\nabla \bu^\eps)\cdot \norm^{\Omega^\eps} = 0 \quad \text{on } \partial\Omega^\eps.
	\end{equation} 	
	\end{itemize} 
\end{prop} 	
\begin{proof} 
	$E^{\lambda, \eps}_{\bg}$ is a proper, convex and coercive functional on $W^{1,2}(\Omega^\eps)$, hence it is weakly lower semicontinuous and attains minimum. By strict convexity, the minimizer $\bu^\eps$ is unique. 
	
	Using convexity of $\Phi$, 
	\[ \Phi(\xi) \leq \varphi_\eps* \Phi(\xi) \leq \max_{\overline{B(\xi, \eps)}} \Phi.\]
	Hence, $\varphi_\eps* \Phi$ is of linear growth. Owing to Proposition \ref{convexprop}, there exists $C>0$, such that 
	\begin{equation} \label{DPhiepsest} 
    	|D\Phi_\eps(\xi)| \leq |D (\varphi_\eps *\Phi)(\xi)| + \eps |\xi| \leq C + \eps |\xi| \quad \text{for } \xi \in \R^m. 
	\end{equation} 
	With this growth condition at hand, one can easily prove that $\bu^\eps$ is a weak solution to the Euler-Lagrange system (\ref{ELeqeps}, \ref{ELbceps}). 
	
	Next, again using Proposition \ref{convexprop}, we obtain,
	\[|D^2\Phi_\eps(\xi)| \leq C(\eps) \quad \text{for } \xi \in \R^m.\]
	By flattening the boundary and applying a variant of tangential difference quotient technique, we then obtain $\bu^\eps \in W^{2,2}(\Omega^\eps)$. We present this argument in detail in the appendix. Consequently, 
\[\nabla(D\Phi_\eps(\nabla \bu^\eps)) = D^2\Phi_\eps(\nabla \bu^\eps)\, \nabla^2 \bu^\eps \in L^2(\Omega^\eps, \R^{m^2}).\]
\end{proof} 
We have following lemmata. 
\begin{lemma}\label{goodseq} 
Let $\widetilde \Psi \colon \R \to [0, +\infty[$ be an even, convex function and let $\bw \in \W$. There exists a family of maps $(\bw^\eps)_{\eps \in ]0, \eps_0]}$ such that $\bw^\eps \in W^{1,\infty}(\Omega^\eps)$, $\bw^\eps \to \bw$ in $\W$ as $\eps \to 0^+$,  
\begin{equation}\label{recseq} 
\lim_{\eps \to 0^+} \int_{\Omega^\eps} \widetilde \Psi (\Phi_\eps(\nabla \bw^\eps)) = \int_{\Omega} \widetilde \Psi (\Phi(\nabla \bw))
\end{equation}
and
\begin{equation}\label{recseq2} 
\lim_{\eps \to 0^+} \int_{\Omega^\eps\setminus \Omega} (\bw^\eps)^2 = 0. 
\end{equation}
\end{lemma} 

\begin{proof} 
    We denote $\Psi = \widetilde \Psi \circ \Phi$, $\Psi_\eps = \widetilde \Psi \circ \Phi_\eps$. Since $\widetilde \Psi$ is convex and non-decreasing on $[0, + \infty[$, both $\Psi$ and $\Psi_\eps$ are convex. Let $\widetilde \bw \in L^2(\Rm) \cap W^{1,1}(\Rm)$ be an extension of $\bw$. Fix $x_0 \in \Omega$. For $\mu > 0$ consider affine dilation $S_{\mu}\colon \R^m \to \R^m$ given by $S_\mu (x) = x_0 + (1+ \mu) (x - x_0)$. Since $\Omega$ is open and convex, we have $\Omega \subset\subset S_\mu(\Omega)$ for any $\mu >0$. We define $\widetilde \bw^\mu \in  L^2(\Rm) \cap W^{1,1}(\Rm)$ by $\widetilde \bw^\mu(x) = \widetilde \bw(S_{\mu}^{-1}(x))$. Due to convexity of $\Psi$, we have for $x \in S_{\mu}(\Omega)$ 
    \begin{equation*}
    \Psi(\nabla \widetilde \bw^{\mu} (x)) = \Psi\left(\tfrac{1}{1+\mu}\nabla \bw (S_{\mu}^{-1}(x))\right) \leq \tfrac{1}{1+\mu} \Psi\left(\nabla \bw (S_\mu^{-1}(x))\right) + \tfrac{\mu}{1+\mu} \Psi(0) 
    \end{equation*} 
    and so 
    \begin{multline} \label{wtildedeltaest}
    \int_{S_\mu(\Omega)} \Psi(\nabla \widetilde \bw^\mu(x))\dd x \leq \int_{S_\mu(\Omega)} \tfrac{1}{1+\mu} \Psi\left(\nabla \bw (S_\mu^{-1}(x))\right) + \tfrac{\mu}{1+\mu} \Psi(0)\dd x\\ 
    = (1+\mu)^{m-1} \int_{\Omega} \Psi(\nabla \bw) + \mu(1+\mu)^{m-1} |\Omega|\Psi(0). 
    \end{multline} 
     For any $\delta >0$, let $\hat \mu = \hat\mu(\delta)$ denote the smallest $\mu>0$ such that $\Omega + B_\delta(0) \subset S_\mu(\Omega)$. It is easy to check that $\hat \mu$ is well defined and $\hat \mu(\delta) \to 0$ as $\delta \to 0^+$. Pick any $\delta_0>0$. Note that we have locally uniform convergence $\Phi_\eps \to \Phi$, and hence also $\Psi_\eps \to \Psi$ as $\eps \to 0^+$. Recalling also that $|\Omega^\eps \setminus \Omega| \to 0^+$, we see that there exists an $\eps_0>0$ such that
 \begin{equation}\label{pr1}
 |{\Omega^\eps}|\cdot \max\left\{\left|\Psi_\eps(\xi) - \Psi(\xi)\right|\colon |\xi| \leq \|\nabla \varphi_{\delta} *\widetilde \bw^{\hat\mu(\delta)}\|_{L^\infty({\Omega^\eps}, \Rm)} \right\} \leq \tfrac{1}{2} \delta
 \end{equation} 
 and 
  \begin{equation}\label{pr2}
 \int_{\Omega^\eps\setminus\Omega} \Psi(\nabla \varphi_{\delta} *\widetilde \bw^{\hat\mu(\delta)}) \leq \tfrac{1}{2} \delta
 \end{equation}
 are satisfied for $\delta=\delta_0$ and all $\eps \in]0, \eps_0]$. We define $ \hat \delta\colon ]0,\eps_0] \to [0, \delta_0]$ by the following formula,
 \[\hat \delta(\eps)= \inf\{\delta \in]0,\delta_0]: (\ref{pr1}, \ref{pr2}) \text{ hold}\},\]
 and set  for $\eps \in]0,\eps_0]$,
 \[\bw^\eps = \left\{ \begin{array}{l} \widetilde \bw |_{\Omega^\eps} \quad \text{if } \hat\delta(\eps) = 0, \\
 \varphi_{\hat\delta(\eps)} * \widetilde \bw^{\hat\mu(\hat\delta(\eps))}\big|_{\Omega^\eps} \quad \text{otherwise}. \end{array} \right.\] 
One can check that if $\hat \delta(\eps) \neq 0$ then (\ref{pr1}, \ref{pr2}) hold with $\delta = \hat \delta(\eps)$, while if $\hat \delta(\eps) = 0$ then  
 \[ \max\left\{\left|\Psi_\eps(A) - \Psi(A)\right|\colon |A| \leq \ess\sup_{\Omega^\eps} |\nabla \widetilde \bw| \right\} = \int_{\Omega^\eps\setminus\Omega} \Psi(\nabla \widetilde \bw) =0.\]
Again, by locally uniform convergence of $\Psi_\eps$ and convergence of $|\Omega^\eps\setminus \Omega|$ to zero, $\hat\delta(\eps)$ tends to $0$ as $\eps\to 0^+$. In particular, $\bw^\eps \to \bw$ in $W^{1,1}(\Omega)$ and $L^2(\Omega)$ as $\eps\to 0^+$. 
 
 By the definitions of $\bw^\eps$ and $\hat\delta$, we have 
 \begin{equation}\label{sbound}\int_{\Omega^\eps} \Psi_\eps(\nabla \bw^\eps) = \int_{\Omega^\eps} \left(\Psi_\eps(\nabla \bw^\eps) - \Psi(\nabla \bw^\eps)\right) + \int_{\Omega^\eps} \Psi(\nabla \bw^\eps) \leq \int_\Omega \Psi(\nabla \bw^\eps) + \hat{\delta}(\eps). 
 \end{equation} 
 Due to convexity of $\Psi$ and $\int_{\R^m} \varphi = 1$, 
  \begin{equation} \label{convPhiest} 
 \int_{\Omega} \Psi(\nabla \bw^\eps) \leq\int_{\Omega} \varphi_{\hat \delta(\eps)} * \left(\Psi\left(\nabla \widetilde \bw^{\hat \mu(\hat \delta(\eps))} \right)\right) \leq \int_{\Omega+B_\delta(0)} \Psi\left(\nabla \widetilde \bw^{\hat \mu(\hat \delta(\eps))} \right)  
 \end{equation} 
provided that $\hat \delta(\eps) \neq 0$. By (\ref{sbound}, \ref{convPhiest}, \ref{wtildedeltaest}), 
 \begin{equation} 
 \limsup_{\eps \to 0^+} \int_{\Omega^\eps} \Psi_\eps(\nabla \bw^\eps) \leq \int_\Omega \Psi(\nabla \bw).
 \end{equation} 
 On the other hand, since $\widetilde \Psi$ is non-decreasing on $[0, +\infty[$ and $\Phi_\eps \geq \Phi$, we have $\Psi_\eps \geq \Psi$. Hence, 
  \begin{equation}\label{lowerest}  \liminf_{\eps \to 0^+}  \int_{\Omega^\eps} \Psi_\eps(\nabla \bw^\eps) \geq \liminf_{\eps \to 0^+}  \int_{\Omega} \Psi(\nabla \bw^\eps) \geq \int_{\Omega} \Psi(\nabla \bw)
  \end{equation} 
   which concludes the proof of \eqref{recseq}. It remains to check \eqref{recseq2}: 
   \begin{multline}\label{recseq2est}\int_{\Omega^\eps\setminus\Omega} (\bw^\eps)^2 \leq \int_{\Omega^\eps \setminus \Omega} \varphi_{\hat\delta(\eps)} * \left(\widetilde \bw^{\hat \mu(\hat \delta(\eps))}\right)^2 \leq \int_{\Omega^\eps \setminus \Omega\, +\, B_{\hat{\delta}(\eps)}(0)} \left(\widetilde \bw^{\hat \mu(\hat \delta(\eps))}\right)^2\\ = \left(1 + \hat \mu(\hat \delta(\eps))\right)^m \int_{S_{\hat \mu(\hat \delta(\eps))}^{-1}\left(\Omega^\eps \setminus \Omega + B_{\hat{\delta}(\eps)}(0)\right)} \widetilde \bw^2. 
   \end{multline} 
   Since $\left|S_{\hat \mu(\hat \delta(\eps))}^{-1}\left(\Omega^\eps \setminus \Omega + B_{\hat{\delta}(\eps)}(0)\right)\right| \to 0$ as $\eps \to 0^+$, the r.\,h.\,s\;of \eqref{recseq2est} converges to $0$. 
 \end{proof} 

\begin{lemma} \label{Gama-zb}
	Let $(\bff^\eps)_{\eps \in ]0, \eps_0]}$ be a family of maps such that $\bff^\eps \in L^2(\Omega^\eps)$ for $\eps \in ]0, \eps_0]$,  $\bff^\eps \to \bff$ in $L^2(\Omega)$ and $\int_{\Omega^\eps\setminus \Omega} (\bff^\eps)^2 \to 0$ as $\eps \to 0^+$. Then $E^{\lambda, \eps}_{\bff^\eps}$ $\Gamma$-converges to $E^\lambda_{\bff}$ as $\eps \to 0^+$ with respect to the weak convergence in $\W$.
\end{lemma} 
\begin{proof} 
 Let $\bw \in \W$. First, take any family $(\bw^\eps)_{\eps \in ]0, \eps_0]}$, $\bw^\eps \in W^{1,2}(\Omega^\eps)$, such that $\bw^\eps \rightharpoonup \bw$ in $\W$ when $\eps \to 0^+$. Since $\Phi$ is convex, we have $\Phi_\eps \geq \Phi$. Hence, due to the weak lower semicontinuity of convex integrals,
 \[ \liminf_{\eps \to 0^+} E^{\lambda, \eps}_{\bff^\eps}(\bw^\eps) \geq \liminf_{\eps \to 0^+} \lambda \int_{\Omega} \Phi(\nabla \bw^\eps) + \frac{1}{2} \int_{\Omega} |\bw^\eps - \bff^\eps|^2 \geq \lambda \int_{\Omega} \Phi(\nabla \bw) + \frac{1}{2} \int_{\Omega} |\bw - \bff|^2.\] 
Thus we have proved the lower bound inequality in the definition of $\Gamma$-convergence. 

On the other hand, given any $\bw \in \W$, let $(\bw^\eps)_{\eps \in ]0, \eps_0]}$ be the family provided by Lemma~\ref{goodseq} given $\widetilde \Psi = |\cdot|$. Then $\bw^\eps \to \bw$ in $W^{1,1}(\Omega)$ and $\int_{\Omega^\eps} \Phi_\eps(\bw^\eps) \to \int_{\Omega} \Phi(\bw)$. Moreover, since $\bff^\eps \to \bff$, $\bw^\eps \to \bw$ in $L^2(\Omega)$ and $\int_{\Omega^\eps\setminus \Omega} (\bff^\eps)^2 \to 0$, $\int_{\Omega^\eps\setminus \Omega} (\bw^\eps)^2 \to 0$ as $\eps \to 0^+$,
\[\int_{\Omega^\eps}|\bw^\eps - \bff^\eps|^2 = \int_{\Omega}|\bw^\eps - \bff^\eps|^2 + \int_{\Omega^\eps\setminus \Omega}|\bw^\eps - \bff^\eps|^2 \to \int_{\Omega}|\bw - \bff|^2 \]
as $\eps \to 0^+$. Thus, $(\bw^\eps)_{\eps \in ]0, \eps_0]}$ is a (generalized) recovery sequence for $\bw$.  
 \end{proof} 

\section{A superlinear estimate} 	
Let us recall a result in linear algebra, which permits us to generalize the results of \cite{NakayasuRybka} to higher dimensions.
\begin{lemma} \label{gal} Let $A$, $B$, $C$ be $m\times m$ symmetric matrices. Suppose that $A$ and $B$ are positive semidefinite. Then $\tr ACBC \geq 0$. \end{lemma} 
\begin{proof} 
	   In the proof we suppress the summation convention. As $A, B \geq 0$, there exist $\lambda_k, \mu_k \in [0,+\infty[$ and $\be_k, \bff_k \in \Rm$, $k=1,\ldots, m$ with 
		   \[ A = \sum_{k=1}^m \lambda_k \be_k \otimes \be_k, \qquad B = \sum_{k=1}^m \mu_k \bff_k \otimes \bff_k. \]  
		   Thus, appealing to symmetry of $C$, we record,
		   \[\tr ACBC = \sum_{k,l=1}^m \lambda_k\; \mu_l\; (\be_k\! \cdot C \bff_l)\; (\bff_l\! \cdot C \be_k) = \sum_{k,l=1}^m \lambda_k\; \mu_l\; (\be_k \cdot C \bff_l)^2 \geq 0. 
		   \]
		   \end{proof} 
Next result is a generalization of \cite[Theorem 3.1]{NakayasuRybka}. Its proof is based on Lemma \ref{gal} and a boundary estimate that relies on convexity of the domain.
\begin{lemma} \label{mainlem} 
	Suppose that $\bg \in W^{1,\infty}(\Omega^\eps)$. Let $\widetilde \Psi\colon \R \to [0,+\infty[$ be an even, convex function and let $\bu^\eps$ be the minimizer of $E^{\lambda, \eps}_{\bg}$. Denote $\Psi_\eps = \widetilde \Psi \circ \Phi_\eps$. Then,
	\begin{equation}\label{superest} 
	\int_{\Omega^\eps} \Psi_\eps(\nabla \bu^\eps) \leq \int_{\Omega^\eps} \Psi_\eps(\nabla \bg). 
	\end{equation} 
\end{lemma} 
\begin{proof} We approximate $\widetilde \Psi$ with a sequence of smooth, even, convex functions of at most linear growth in the following way. For $k \in \N$, we define a convex $T_k \widetilde \Psi \colon \R \to [0, +\infty[$ by
 \[ T_k \widetilde \Psi(p) = \widetilde \Psi(p) \text{ if } |\widetilde \Psi'(p)| \leq k, \quad |(T_k \widetilde \Psi)'(p)| = k\text{ if } |\widetilde \Psi'(p)| > k \]
for a.\,e.\;$p \in \R$ where $\widetilde \Psi'(p)$ exists. Next, for $k \in \N$ we set
\[\widetilde{\Psi}_k = \varphi_{\frac{1}{k}} * (T_k \widetilde \Psi), \]
where $(\varphi_{\delta})_{\delta>0}$ is a standard approximate identity on the line, and
\[\Psi_{\eps,k} = \widetilde{\Psi}_k \circ \Phi_\eps.\]
Clearly, $\Psi_{\eps,k}$ is a smooth, even, convex function for any $\eps >0$, $k \in \N$ and 
$\Psi_{\eps,k} \to \Psi_\eps$ locally uniformly when $k \to +\infty$. We calculate
\begin{equation}\label{truncderiv1} 
D\Psi_{\eps,k}(\nabla \bu^\eps) = \widetilde{\Psi}_k'(\Phi_{\eps}(\nabla \bu^\eps)) D\Phi_{\eps}(\nabla \bu^\eps), 
\end{equation} 
\begin{equation}\label{truncderiv2} 
D^2\Psi_{\eps,k}(\nabla \bu^\eps) = \widetilde{\Psi}_k''(\Phi_{\eps}(\nabla \bu^\eps)) D\Phi_{\eps}(\nabla \bu^\eps) \otimes D\Phi_{\eps}(\nabla \bu^\eps)  + \widetilde{\Psi}_k'(\Phi_{\eps}(\nabla \bu^\eps)) D^2\Phi_{\eps}(\nabla \bu^\eps). 
\end{equation}
Appealing to Proposition \ref{convexprop} we get $D^2\Psi_{\eps,k}(\nabla \bu^\eps) \in L^\infty(\Omega^\eps, \R^{m^2})$, and by Proposition \ref{classic} 
\[\nabla(D\Psi_{\eps,k}(\nabla \bu^\eps)) = D^2\Psi_{\eps,k}(\nabla \bu^\eps)\, \nabla^2 \bu^\eps \in L^2(\Omega^\eps, \R^{m^2}). \]   
It also follows from Proposition \ref{classic} that $\div(D \Phi_\eps(\nabla \bu^\eps)) = \frac{1}{\lambda}(\bu^\eps - \bg) \in W^{1,2}(\Omega^\eps)$. Thus,
\begin{equation} \label{divintegr} 
\nabla \div(D \Phi_\eps(\nabla \bu^\eps))  = \left((D \Phi_\eps(\nabla \bu^\eps))^\alpha_{i,x_i x_k}\right) =  \div( \nabla D \Phi_\eps(\nabla \bu^\eps)) \in L^2(\Omega^\eps, \Rm).
\end{equation}
Hence, we can calculate \cite[Lemma 1]{FujiwaraMorimoto},  
\begin{multline} \label{maincomp} 
\int_{\Omega^\eps} \nabla \div(D \Phi_\eps(\nabla \bu^\eps)) \cdot D \Psi_{\eps, k}(\nabla \bu^\eps) + \int_{\Omega^\eps} \nabla (D \Phi_\eps(\nabla \bu^\eps)) \cddot \nabla (D \Psi_{\eps, k}(\nabla \bu^\eps)) \\ = \langle \nabla D \Phi_\eps(\nabla \bu^\eps) \cdot \norm^{\Omega^\eps}, D \Psi_{\eps, k}(\nabla \bu^\eps)\rangle_{H^{-\frac{1}{2}}(\partial {\Omega^\eps}, \Rm), H^{\frac{1}{2}}(\partial {\Omega^\eps}, \Rm)}. 
\end{multline} 
We have 
\begin{multline} \label{nowe}
\nabla (D \Phi_\eps(\nabla \bu^\eps)) \cddot \nabla (D \Psi_{\eps, k}(\nabla \bu^\eps)) = (D^2 \Phi_\eps(\nabla \bu^\eps) \cdot \nabla^2 \bu^\eps) \cddot (D^2 \Psi_{\eps, k}(\nabla \bu^\eps) \cdot \nabla^2 \bu^\eps)
\\=  D^2 \Phi_\eps(\nabla \bu^\eps)_{ij}\, (\bu^\eps)_{x_j x_k} \, D^2 \Psi_{\eps, k}(\nabla \bu^\eps)_{kl}\,(\bu^\eps)_{x_l x_i}. 
\end{multline} 
This expression is of form $\tr ACBC$ with $A,B,C$ satisfying conditions of Lemma \ref{gal}.
Therefore, 
\begin{equation} \label{ellipticineq} 
\int_{\Omega^\eps} \nabla (D \Phi_\eps(\nabla \bu^\eps)) \cddot \nabla (D \Psi_{\eps, k}(\nabla \bu^\eps)) \geq 0. 
\end{equation} 

\begin{remark}\label{uwaga}
At this point we would like to explain why our approach fails in the vectorial setting. In that case, instead of (\ref{nowe}), we would have to deal with
\begin{multline*} 
\nabla (D \Phi_\eps(\nabla \bu^\eps)) \cdddot \nabla (D \Psi_{\eps, k}(\nabla \bu^\eps)) = (D^2 \Phi_\eps(\nabla \bu^\eps) \cddot \nabla^2 \bu^\eps) \cdddot (D^2 \Psi_{\eps, k}(\nabla \bu^\eps) \cddot \nabla^2 \bu^\eps)
\\=  D^2 \Phi_\eps(\nabla \bu^\eps)^{\alpha \beta}_{ij}\, (\bu^\eps)^\beta_{x_j x_k} \, D^2 \Psi_{\eps, k}(\nabla \bu^\eps)^{\alpha \gamma}_{kl}\,(\bu^\eps)^\gamma_{x_l x_i}
\end{multline*} 
which is a sum of expressions of form $\tr A^{\alpha \beta} C^{\beta} B^{\alpha \gamma}C^\gamma$, where $A^{\alpha \beta},B^{\alpha \beta},C^{\beta}$ are $m\times m$ matrices for $\alpha, \beta = 1, \cdots, n$. Unfortunately, the conditions of Lemma \ref{gal} are in general not satisfied, as $A^{\alpha \beta}$, $B^{\alpha \beta}$ may fail to be non-negative definite if $\alpha \neq \beta$. 
\end{remark}

Now we return to the proof of Lemma \eqref{mainlem}. Let $(\varphi_l)_{l\in \mathbb{N}} \subset C^\infty(\overline{{\Omega^\eps}}, \Rm)$ be such that
\begin{equation}\label{Neumapprox} 
\varphi_l\cdot \norm^{\Omega^\eps} = 0 \quad \text{on } \partial {\Omega^\eps}, 
\end{equation} 
\begin{equation}\label{Napproxconv1} 
\varphi_l \to D \Phi_\eps(\nabla \bu^\eps)\quad  \text{in } W^{1,2}({\Omega^\eps}, \Rm),
\end{equation}
\begin{equation}\label{Napproxconv2}
 \nabla \div \varphi_l \to \nabla \div (D \Phi_\eps(\nabla \bu^\eps)) \quad \text{in } L^2({\Omega^\eps}, \Rm). 
\end{equation} 
Such a sequence can be produced by flattening the boundary, even reflection and mollification of the pushforward of $D \Phi_\eps(\nabla \bu^\eps)$ by the flattening diffeomorphism. Furthermore, let $\bar{\norm}^{\Omega^\eps} \in C^\infty(\R^m, \R^m)$ be an extension of $\norm^{\Omega^\eps}$ that is constant on the fibers of a tubular neighborhood of $\partial {\Omega^\eps}$. By Leibniz' rule 
\begin{equation}\label{boundarycalc}
(\nabla \varphi_l) \cdot \bar{\norm}^{\Omega^\eps} = \nabla (\varphi_l \cdot \bar{\norm}^{\Omega^\eps}) - \varphi_l \cdot (\nabla \bar{\norm}^{\Omega^\eps}).
\end{equation}
Due to \eqref{Neumapprox}, $\nabla (\varphi_l \cdot \bar{\norm}^{\Omega^\eps})$ is perpendicular to $\partial {\Omega^\eps}$ on $\partial {\Omega^\eps}$. On the other hand, 
\[D\Psi_{\eps, k}(\nabla \bu^\eps) \cdot \norm^{\Omega^\eps} = \widetilde \Psi_k'(\Phi_\eps(\nabla \bu^\eps)) D\Phi_\eps(\nabla \bu^\eps) \cdot \norm^{\Omega^\eps} = 0 \quad \text{on } \partial {\Omega^\eps}\]
because of  \eqref{ELbceps}. Therefore, 
\[\nabla (\varphi_l \cdot \bar{\norm}^{\Omega^\eps}) \cdot D \Psi_\eps (\nabla \bu^\eps) = 0 \quad \text{on } \partial {\Omega^\eps}\]
and, by virtue of \eqref{boundarycalc}, 
\begin{equation*}
\label{boundarycalc2} 
\int_{\partial {\Omega^\eps}} D \Psi_\eps (\nabla \bu^\eps) \cdot \nabla \varphi_l \cdot \bar{\norm}^{\Omega^\eps} = - \int_{\partial {\Omega^\eps}} D \Psi_\eps (\nabla \bu^\eps) \cdot \nabla \bar{\norm}^{\Omega^\eps} \cdot \varphi_l.
\end{equation*} 
Passing with $l \to + \infty$, by continuity of the trace operator, we obtain
\begin{multline*} \langle \nabla D \Phi_\eps(\nabla \bu^\eps) \cdot \norm^{\Omega^\eps}, D \Psi_{\eps, k}(\nabla \bu^\eps)\rangle_{H^{-\frac{1}{2}}(\partial {\Omega^\eps}, \Rm), H^{\frac{1}{2}}(\partial {\Omega^\eps}, \Rm)} \\ = - \int_{\partial {\Omega^\eps}} D \Psi_{\eps, k} (\nabla \bu^\eps) \cdot \nabla \bar{\norm}^{\Omega^\eps} \cdot D\Phi_\eps (\nabla \bu^\eps) \\ = - \int_{\partial {\Omega^\eps}} \widetilde \Psi_k'(\Phi_\eps(\nabla \bu^\eps)) D\Phi_\eps(\nabla \bu^\eps) \cdot \nabla \bar{\norm}^{\Omega^\eps} \cdot D\Phi_\eps (\nabla \bu^\eps). 
\end{multline*} 
We observe that 
\begin{equation*} 
D\Phi_\eps(\nabla \bu^\eps) \cdot \nabla \bar{\norm}^{\Omega^\eps} \cdot D\Phi_\eps (\nabla \bu^\eps) = \mathcal A (D\Phi_\eps(\nabla \bu^\eps), D\Phi_\eps(\nabla \bu^\eps)), 
\end{equation*} 
where we have denoted by $\mathcal A$ the classical second fundamental form of hypersurface $\partial {\Omega^\eps}$. Since ${\Omega^\eps}$ is convex, $\mathcal A$ is non-negative. We obtain 
\begin{equation} \label{boundarysign} 
\langle \nabla D \Phi_\eps(\nabla \bu^\eps) \cdot \norm^{\Omega^\eps}, D \Psi_{\eps, k}(\nabla \bu^\eps)\rangle_{H^{-\frac{1}{2}}(\partial {\Omega^\eps}, \Rm), H^{\frac{1}{2}}(\partial {\Omega^\eps}, \Rm)} \leq 0. 
\end{equation} 
Combining (\ref{maincomp}, \ref{ellipticineq}, \ref{boundarysign}) and applying convexity of $\Psi_\eps$ yields 
\begin{multline}\label{boundepsk}
\int_{\Omega^\eps} \Psi_{\eps, k}(\nabla \bu^\eps) - \int_{\Omega^\eps} \Psi_{\eps, k}(\nabla \bg) \leq \int_{\Omega^\eps} D \Psi_{\eps, k}(\nabla \bu^\eps) \cdot (\nabla \bu^\eps - \nabla \bg) \\= \lambda\int_{\Omega^\eps} \nabla \div(D \Phi_{\eps, k}(\nabla \bu^\eps)) \cdot D \Psi_{\eps, k}(\nabla \bu^\eps) \leq 0. 
\end{multline}  
Finally, we pass to the limit $k\to +\infty$ with \eqref{boundepsk} using the monotone convergence theorem.
\end{proof} 

\section{Proof of Theorem \ref{w11thm}}
We want to exhibit the minimizer of $E^\lambda_{\bff}$ as a weak limit in $\W$ of a sequence of minimizers of $E^{\lambda, \eps}_{\bff^\eps}$ with suitably chosen $(\bff^\eps)_{\eps \in ]0, \eps_0]}$. For that purpose, we need the following characterization of weak compactness in $L^1(\Omega, \R^N)$.  
\begin{thm}\label{Pettis}
Let $\mathcal F \subset L^1(\Omega, \RN)$. The following conditions are equivalent: 
\begin{itemize} 
\item[(a)] $\mathcal F$ is (sequentially) weakly relatively compact, 
\item[(b)] $\mathcal F$ is uniformly integrable, 
\item[(c)] there exists an even, convex function $\widetilde \Psi \colon \R \to [0, +\infty[$ and $C>0$ such that 
\[\lim_{|p| \to + \infty} \frac{\widetilde \Psi(p)}{|p|} = + \infty \quad \text{and} \quad \int_\Omega \widetilde \Psi(|\bw|) \leq C \text{ for all } \bw \in \mathcal F.\]
\end{itemize} 
\end{thm} 
The equivalence {\it (a)}$\iff${\it (b)} is the content of the Dunford-Pettis theorem. The equivalence {\it (b)}$\iff${\it (c)} is due to de la Vall\' ee Poussin, see \cite[1.2]{RaoRen}. Note that weak compactness and sequential weak compactness are equivalent in Banach spaces (this is the Eberlein-Shmulyan theorem). The following observation \cite[1.2, Corollary 3]{RaoRen} is a very useful immediate consequence of Theorem \ref{Pettis}.  
\begin{cor} \label{superPsi} Let $w \in L^1(\Omega, [0,+\infty[)$. There exists an even, convex function $\widetilde \Psi \colon \R\to [0, +\infty[$ such that 
	\[\lim_{|p| \to + \infty} \frac{\widetilde \Psi(p)}{|p|} = + \infty,  \qquad \int_\Omega \widetilde \Psi(w) < +\infty.\]
	\qed
\end{cor}

Let $\widetilde \Psi \colon \R\to [0, +\infty[$ be an even, convex function such that $\lim_{|p| \to + \infty} \frac{\widetilde \Psi(p)}{|p|} = + \infty$ and $\int_\Omega \widetilde \Psi(\Phi(\nabla \bff)) < +\infty$. Let $(\bff^\eps)_{\eps \in ]0, \eps_0]}$ be the family provided by Lemma \ref{goodseq} given $\bw = \bff$. Using Lemma \ref{mainlem} and \eqref{recseq}, recalling that $\Phi_\eps \geq \Phi$, we deduce 
\[\int_{\Omega} \widetilde{\Psi}(\Phi(\nabla \bu^\eps))\leq \int_{\Omega^\eps} \widetilde{\Psi}(\Phi_\eps(\nabla \bu^\eps))\leq  \int_\Omega \widetilde{\Psi}(\Phi(\nabla \bff))\;+\; 1 \]
for small enough $\eps$. By growth condition \eqref{linear}, we obtain a uniform bound
\begin{equation}\label{superbound}
\int_{\Omega} \widetilde{\Psi}(C_1|\nabla \bu^\eps|)\leq \int_\Omega \widetilde{\Psi}(\Phi(\nabla \bff))\;+\; 1 .
\end{equation} 
From $E^{\lambda, \eps}_{\bff^\eps}(\bu^\eps)\leq E^{\lambda, \eps}_{\bff^\eps}( 0)$, we get also 
\begin{equation}\label{l2bound}
\int_{\Omega} |\bu^\eps|^2 \leq 4 \int_{\Omega^\eps} |\bff^\eps|^2\;+\; 1
\end{equation} 
for small enough $\eps$. 
Invoking Theorem \ref{Pettis}, we deduce from \eqref{superbound} and \eqref{l2bound} the existence of $\bu \in \W$ and a sequence $(\eps_k)_{k\in \mathbb{N}}$, $\eps_k \to 0$ as $k \to +\infty$, such that 
\begin{equation}\label{weakconv} 
\bu^{\eps_k} \rightharpoonup \bu \quad \text{in } \W. 
\end{equation}
We recall that Lemma \ref{Gama-zb} yields $\Gamma$-convergence of $E^{\lambda,\varepsilon}_{\bff^\eps}$ to $E^{\lambda}_{\bff}$ with respect to the weak convergence in $\W$. Thus, we deduce that $u$ is a minimizer of $E^{\lambda}_{\bff}$. 

Now, let $\widetilde \Psi\colon \R \to [0, +\infty[$ be any even, convex function and let $(\bff^\eps)_{\eps \in ]0, \eps_0]}$ be the family produced by Lemma \ref{goodseq} given $\widetilde \Psi$ and $\bw = \bff$. We recall that by Lemma \ref{mainlem},
\[\int_{\Omega} \widetilde{\Psi}(\Phi(\nabla \bu^\eps))\leq \int_{\Omega^\eps} \widetilde{\Psi}(\Phi_{\eps}(\nabla \bu^{\eps}))\leq \int_{\Omega^\eps} \widetilde{\Psi}(\Phi_{\eps}(\nabla \bff^\eps)),\]
whence \eqref{Psiest} follows by weak convergence of $\bu^{\eps_k}$ and \eqref{recseq}. 
\qed

\section{Proof of Theorem \ref{Dsuthm}} 
Let $\widetilde \Psi\colon \R \to [0, +\infty[$ be an even, convex function of at most linear growth. We introduce notation: $\Psi_\eps = \widetilde \Psi \circ \Phi_\eps$, $\Psi = \widetilde \Psi \circ \Phi$, $\Psi^\infty(\xi) = \lim_{t \to + \infty} \frac{\Psi(t \xi)}{t}$ for $\xi \in \Rm$ and 
\[\overline{F}_\Psi(\bw) = \int_{\Omega} \Psi (\nabla^{ac} \bw) + \int_{\Omega} \Psi^\infty \left(\tfrac{\nabla^{s} \bw}{|\nabla^{s} \bw|}\right) \dd |\nabla^{s} \bw| \]
for $\bw \in BV(\Omega)$. The functional $\overline{F}_\Psi$ is weakly-$*$ lower semicontinuous \cite{GoffmanSerrin}, see also \cite[Theorem 5.47]{AmbrosioFuscoPallara}. 

We now give $BV$ variants of Lemmata \ref{goodseq} and \ref{Gama-zb}.

\begin{lemma}\label{goodseqbv} 
Let $\widetilde \Psi \colon \R \to [0, +\infty[$ be an even, convex function of at most linear growth and let $\bw \in BV(\Omega) \cap L^2(\Omega)$. There exists a family of maps $(\bw^\eps)_{\eps \in ]0, \eps_0]}$ such that $\bw^\eps \in W^{1,\infty}(\Omega^\eps)$, $\bw^\eps \to \bw$ in $L^2(\Omega)$ and weakly-$*$ in $BV(\Omega)$ as $\eps \to 0^+$,   
\begin{equation}\label{recseqbv} 
\lim_{\eps \to 0^+} \int_{\Omega^\eps} \Psi_\eps (\nabla \bw^\eps) = \overline{F}_\Psi(\bw)
\end{equation}
and
	\begin{equation}\label{recseq2bv} 
		\lim_{\eps \to 0^+} \int_{\Omega^\eps\setminus \Omega} (\bw^\eps)^2 = 0. 
\end{equation}
\end{lemma} 
\begin{proof} 
We construct the sequence $\bw^\eps$ as in Lemma \ref{goodseq}. The proof that it satisfies our assertions also follows along the same lines. The important changes are: 
\begin{itemize} 
\item $\Psi(\nabla \bw)$ has to be understood as the measure $\Psi (\nabla^{ac} \bw) \Lb^m + \Psi^\infty \left(\frac{\nabla^{s} \bw}{|\nabla^{s} \bw|}\right) |\nabla^{s} \bw|$ (and $\Psi(\nabla \widetilde \bw)$, $\Psi(\nabla \widetilde \bw^\mu)$ likewise), 
\item to obtain inequality \eqref{convPhiest}, we apply \cite[Lemma 2.2]{DemengelTemam} to the function $\Psi - \Psi(0)$, 
\item in \eqref{lowerest} we use weak-$*$ lower semicontinuity of $\overline{F}$ on $BV(\Omega)$ (recall \eqref{overlineF}). 
\end{itemize} 
\end{proof} 

\begin{lemma} \label{Gama-zbbv}
	Let $(\bff^\eps)_{\eps \in ]0, \eps_0]}$ be a family of maps such that $\bff^\eps \in L^2(\Omega^\eps)$ for $\eps \in ]0, \eps_0]$, $\bff^\eps \to \bff$ in $L^2(\Omega)$ and $\int_{\Omega^\eps\setminus \Omega} (\bff^\eps)^2 \to 0$ as $\eps \to 0^+$. Then $E^{\lambda, \eps}_{\bff^\eps}$ $\Gamma$-converges to $\overline{E}^\lambda_{\bff}$ as $\eps \to 0^+$ with respect to the weak-$*$ convergence in $BV(\Omega)$ and $L^2(\Omega)$.
\end{lemma} 
\begin{proof} 
The proof is the same as in the case of Lemma \ref{Gama-zb}, except that we need to use weak-$*$ lower semicontinuity of $\overline{F}$ on $BV(\Omega)$ and Lemma \ref{goodseqbv}. 
\end{proof} 

Now, given $l \in \N$, let $\widetilde \Psi(p) = (|p|-l)_+$ for $p \in \mathbb R$. Let $(\bff^\eps)_{\eps \in ]0, \eps_0]}$ be the family provided by Lemma \ref{goodseqbv} given $\bw = \bff$. We denote $\bu^\eps$ the minimizer of $E^{\lambda, \eps}_{\bff^\eps}$. Using $E^{\lambda, \eps}_{\bff^\eps}(\bu^\eps) \leq E^{\lambda, \eps}_{\bff^\eps}(0)$ and \eqref{linear} we get
\[ \lambda C_1 \int_\Omega |\nabla \bu^\eps| + \tfrac{1}{4} \int_\Omega |\bu^\eps|^2 \leq \int_\Omega |\bff^\eps|^2.\]
As the r.\,h.\,s\;is bounded, there exists $\bu \in BV(\Omega) \cap L^2(\Omega)$ and a sequence $(\eps_k)_{k \in \N}$, $\eps_k \to 0^+$ such that $\bu^{\eps_k}$ converges weakly-$*$ in $BV(\Omega)$ and $L^2(\Omega)$. Due to $\Gamma$-convergence (Lemma \ref{Gama-zbbv}) the limit $\bu$ is the minimizer of $\overline{E}^\lambda_{\bff}$.

We recall that Lemma \ref{mainlem} yields 
\[\int_{\Omega^{\eps_k}} \Psi_{\eps_k}(\nabla \bu^{\eps_k}) \leq \int_{\Omega^{\eps_k}} \Psi_{\eps_k}(\nabla \bff^{\eps_k})\] 
for $k \in \N$. Passing to the limit $k\to \infty$, by weak-$*$ lower semicontinuity of $\overline{F}_\Psi$ and \eqref{recseqbv} we obtain 
\begin{equation} \label{overlineFbound} 
\overline{F}_\Psi(\bu) \leq \overline{F}_\Psi(\bff).
\end{equation}
We note that $\Psi^\infty = \Phi^\infty$. Therefore, \eqref{overlineFbound} translates to 
\begin{equation*}
    \int_\Omega \left(\Phi(\nabla^{ac} \bu)- l\right)_+ +\int_\Omega \Phi^\infty\left(\tfrac{\nabla^s \bu}{|\nabla^s \bu|}\right) \dd |\nabla^s \bu| \leq \int_\Omega \left(\Phi(\nabla^{ac} \bff)- l\right)_+ +\int_\Omega \Phi^\infty\left(\tfrac{\nabla^s \bff}{|\nabla^s \bff|}\right) \dd |\nabla^s \bff|.
\end{equation*}
We pass to the limit $l \to + \infty$. 
\qed

\section{Proof of Theorem \ref{gradflow}} 

Given $\bu_0 \in L^2(\Omega) \cap BV(\Omega)$, let $\bu \in W^{1,2}(0, \infty; L^2(\Omega))$ be the solution to the initial value problem $\bu_t \in - \partial \overline{F}(\bu)$ for a.\,e.\;$t>0$, $\bu(0) = \bu_0$. We recall that $\bu$ is given by the nonlinear exponential formula \cite[Corollary 4.4]{maximaux}
\begin{equation}\label{expformula} \bu(t) = \lim_{n\to +\infty} \left(\mathrm{id} + \tfrac{t}{n}\partial \overline F\right)^{-n} \bu_0.
\end{equation} 
This limit is understood in $L^2(\Omega)$ or equivalently in weak-$*$ convergence of $BV(\Omega)$ (as the sequence is uniformly bounded in $BV(\Omega)$, see below). Denoting 
\[\bu^{n,k}(t) = \left(\mathrm{id} + \tfrac{t}{n}\partial \overline F\right)^{-k} \bu_0\] 
we have 
\[\bu^{n,k}(t) + \tfrac{t}{n} \partial \overline F(\bu^{n,k}(t)) \ni \bu^{n, k-1}(t)\]
for $k=1,\ldots, n$, $t>0$. Equivalently, $\bu^{n,k}(t)$ is the minimizer of $\overline E^{\tfrac{t}{n}}_{\bu^{n,k-1}(t)}$. With the notation from previous section, we have for all $n= 1, 2, \ldots$, $k=1, \ldots, n$, $t>0$, $\overline{F}_\Psi(\bu^{n,k}(t)) \leq \overline{F}_\Psi(\bu_0)$
and therefore, by weak-$*$ lower semicontinuity of $\overline{F}_\Psi$, 
\begin{equation*}
    \overline{F}_\Psi(\bu(t)) \leq \overline{F}_\Psi(\bu_0).  
\end{equation*}
Recalling that $\Psi(\xi) = (\Phi(\xi) - l)_+$ and passing to the limit $l \to + \infty$ we recover \eqref{Dsuestgf}. 

Now, suppose that $\bu_0 \in W^{1,1}(\Omega)$. By Corollary \ref{superPsi}, there exists a convex function $\widetilde \Psi$ of superlinear growth such that $\int_\Omega \widetilde \Psi(\Phi(\bu_0))< +\infty$. Then, by Theorem \ref{w11thm}, for all $n= 1, 2, \ldots$, $k=1, \ldots, n$, $t>0$, we have $\bu^{n,k}(t) \in W^{1,1}(\Omega)$ and 
\[\int_\Omega \widetilde \Psi(\Phi(\nabla \bu^{n,k}(t))) \leq \int_\Omega \widetilde \Psi(\Phi(\nabla \bu_0)).\]
By Theorem \ref{Pettis} and \eqref{expformula}, we obtain that $\bu(t) \in W^{1,1}(\Omega)$ for $t>0$ and the convergence in \eqref{expformula} can be upgraded to weak $W^{1,1}(\Omega)$ convergence. Consequently, 
\[\int_\Omega \widetilde \Psi(\Phi(\nabla \bu(t))) \leq \int_\Omega \widetilde \Psi(\Phi(\nabla \bu_0))\] 
for any even, convex $\widetilde \Psi \colon \R \to [0, +\infty[$.  

\section*{Appendix: Second derivatives for the approximate problem.}
Even though Theorems \ref{w11thm}-\ref{gradflow} deal only with the scalar case, we present the regularity result below in vectorial setting.

Let $\Omega \subset \R^m$ be a $C^2$ bounded domain, let $\lambda >0$ and $\bff \in L^2(\Omega, \R^n)$. We consider here the functional $F^\lambda_{\bff}$ on $W^{1,2}(\Omega, \R^n)$ given by  
\[ F^\lambda_{\bff}(\bu) =  \lambda \int_\Omega \Phi(\nabla \bu) + \frac {1}{2}\int_\Omega|\bu - \bff|^2,\]
where $\Phi \in C^2(\Rnm)$ is uniformly convex, i.\,e.\;there exists $\mu >0$ such that
\begin{equation}\label{unifconv}\tfrac{1}{\mu} I^{n\times m} \leq D^2\Phi(A) \leq \mu I^{n\times m} \quad \text{for } A \in \Rnm. 
\end{equation}
We have denoted by $I^{n\times m}$ the identity matrix on $\Rnm$. Possibly enlarging $\mu$, we will also assume 
\begin{equation} \label{DPhi0} 
|D\Phi(0)| \leq \mu. 
\end{equation}
\begin{prop} Let $\bu \in W^{1,2}(\Omega, \R^n)$ be the minimizer of $F^\lambda_{\bff}$. Then, $\bu \in W^{2,2}(\Omega, \R^n)$. 
\end{prop}  
\begin{proof} 
The proof employs the usual difference quotient technique. As in \cite{GiustiMiranda1968}, we prove that $\bu \in W^{2,2}_{loc}(\Omega, \R^n)$. However, as far as the boundary regularity is concerned, we failed to find a satisfactory reference. The treatments presented in \cite[6.3.2]{evans} and \cite[8.4]{Giusti2003} are the closest to our needs that we know of. In the former, general linear elliptic equation is handled, while in the latter quasilinear elliptic equation of form $\div A(\nabla u) = 0$ is considered. In both cases the equation is supplemented with homogeneous Dirichlet boundary condition, but the same proofs work with homogeneous Neumann condition. However, in \cite[6.4]{Giusti2003} only half-ball estimates are obtained. In the case of arbitrary $\Omega$, the need to flatten the boundary complicates the situation, since after the change of variables the form of equation changes. Considering vector-valued $\bu$ introduces further technical difficulty. For these reasons, we include here the complete proof of integrability of the second derivative up to the boundary.\footnote{It has been pointed out by an anonymous referee, that with some additional effort this result could be inferred from \cite{Hamburger}. However, we still believe that it is worthwhile to include the detailed proof here.}

	We recall that $\bu$ satisfies the Euler-Lagrange system 
	\begin{equation} \label{ELeqA} 
\bu - \bff = \lambda \div(D \Phi(\nabla \bu)) \quad \text{in } \Omega, 
\end{equation} 
\begin{equation} \label{ELbcA} 
D \Phi(\nabla \bu)\cdot \norm^\Omega = 0 \quad \text{on } \partial\Omega
\end{equation}
 in a weak sense. In other words,
 \begin{equation}\label{pr3}
  \int_\Omega \lambda D \Phi(\nabla \bu)\nabla \varphi  +\int_\Omega(\bu - \bff)\varphi =0   
 \end{equation}
 holds for all test functions $\varphi\in W^{1,2}(\Omega,\mathbb{R}^n)$. 
	
Let $x_0 \in \partial \Omega$, $r>0$ be such that $\partial \Omega \cap B_{4 r}(x_0)$ coincides (up to isometry) with a graph of a $C^2$ function. In this case, there exists an open set $U \subset \R^m$ and a $C^2$ diffeomorphism $S$ on $B_{4 r}(x_0)$ that maps $B_{3 r}(x_0)$ onto $U$ and $\Omega \cap B_{3 r}(x_0)$ onto $U^+ = \{(y^1, \ldots, y^m) \in U \colon y^m >0\}$. Furthermore, $\det DS = 1$ in $B_{4 r}(x_0)$ \cite[Appendix C.1]{evans}. For $y \in U$, we denote $Q(y) = DS(S^{-1}(y))$. This defines a function $Q \in C^1(\overline{U},SL(m))$. 

For $0< s \leq 3r$, we write $U_s = S(B_s(x_0))$, $U_s^+ = U_s \cap U^+$ and we set $U^0 = \{x \in U \colon x^m = 0 \}$, which coincides with $\partial U^+ \cap U$. We denote 
\begin{equation}\label{CQ}C_Q = \max \left(\|Q\|_{C^1\left(\overline{U_{3r}}, SL(m)\right)}, \sup_{y \in S\left(\overline{U_{3r}}\right)} |Q(y)^{-1}|\right).
\end{equation} 
Furthermore, we define $\widetilde \bff \in L^2(U^+)$ by $\widetilde \bff(y) = \bff(S^{-1}(y))$ for $y \in U^+$ and $\widetilde \bu \in W^{1,2}(U^+)$ by $\widetilde \bu(y) = \bu(S^{-1}(y))$ for $y \in U^+$. 

Let us take any test function $\psi\in W^{1,2}_0(B_{3r})$ in (\ref{pr3}). After performing the change of variables $y=S(x)$ and taking into account that $\det DS(x) =1$ we reach
\begin{equation}\label{pr4}
-\int_{U^+} \lambda  D \Phi(Q\, \nabla \widetilde\bu) \cdot Q \nabla \widetilde\psi
= \int_{U^+}(\widetilde \bu - \widetilde\bff)\widetilde \psi,
\end{equation}
where $\widetilde\psi(y) = \psi(S^{-1}(y))$. Since $S$ is a diffeomorphism, $\psi$ is in fact any test function from $W^{1,2}(U^+)$ vanishing on $\partial U^+ \setminus U^0$.

Now, for $i=1,\ldots, m-1$ and $h \in \mathbb R$, $h \neq 0$, we denote the operator of difference quotient in direction $e_i$ by $\partial^h_i$, i.\,e.
\[(\partial^h_i g)(y) = \frac{g(y + h e_i) - g(y)}{h}\] 
for any function $g$ on $U^+$ and $y \in U^+$ such that $\dist(y, \partial U^+ \setminus U^0) < h$. We will use the following version of the integration by parts formula for the operator $\partial^h_i$,
\begin{equation}\label{pr5}
\int_{U^+} f \partial^h_i g  = \int_{U^+} \partial^{-h}_i f g ,
\end{equation}
    which is valid whenever the support of $f$ or $g$ is at a distance at least $h$ from $\partial U^+ \setminus U^0$.

We take $\varphi \in C^1_c(U, [0,1])$ such that $\varphi = 1$ on $U_r$, $\varphi = 0$ on $U\setminus U_{2r}$, and $h \neq 0$, $|h|<\tfrac{1}{2}\dist(U_{2r}, \partial U_{3r})$. We note that $ \partial^{-h}_i(\varphi^2  \partial^{h}_i\widetilde u)$, $i = 1, \ldots, m-1$ are legitimate test functions for \eqref{pr4} (the summation convention is suppressed here and in the following calculations). If we stick them in (\ref{pr4}) and use (\ref{pr5}), we shall see that the r.\,h.\,s.\; of (\ref{pr4}) takes the following form, 
\begin{equation}\label{diffquoteq1}
 \int_{U^+} |\partial^h_i \widetilde \bu|^2 \,\varphi^2 - \int_{U^+}  \widetilde \bff \partial^{-h}_i ( \varphi^2\partial^h_i \widetilde \bu) =:I_1 - I_2 .
\end{equation}
At the same time, (\ref{pr5}) applied to the l.\,h.\,s.\;of (\ref{pr4}) yields
\begin{equation}\label{diffquoteq2}
- \lambda \int_{U^+}  \varphi \nabla \varphi \cdot \partial^h_i (Q^T D \Phi(Q\, \nabla \widetilde\bu)) \cdot \partial^h_i \widetilde \bu -\lambda \int_{U^+}  \partial^h_i (Q^T D \Phi(Q\, \nabla \widetilde\bu)) \cddot \partial^h_i \nabla \widetilde \bu \, \varphi^2  =:
-\lambda I_3 - \lambda I_4.
\end{equation}
Since $\|\partial^{-h}_i ( \varphi^2\partial^h_i \widetilde \bu)\|_{L^2(U^+,\mathbb{R}^n)}\le \|( \varphi^2\partial^h_i \widetilde \bu)_{y_i}\|_{L^2(U^+,\mathbb{R}^n)}$,  we have 
\begin{equation*} 
|I_2| \leq 
\| \bff\|_{L^2(U^+, \Rn)}\left\|(\varphi^2 \partial^h_i \widetilde \bu)_{y_i}\right\|_{L^2(U^+, \Rn)} \leq \| \bff\|_{L^2(U^+, \Rn)}\left\|\nabla(\varphi^2 \partial^h_i \widetilde \bu)\right\|_{L^2(U^+, \Rnm)}.
\end{equation*} 
Moreover, 
\begin{multline*} 
\left\|\nabla(\varphi^2 \partial^h_i \widetilde \bu)\right\|_{L^2(U^+, \Rnm)}  \leq \left\|\varphi^2 \partial^h_i \nabla \widetilde \bu\right\|_{L^2(U^+, \Rnm)} 
+ 2 \left\| \varphi \nabla \varphi \otimes \partial^h_i \widetilde \bu \right\|_{L^2(U^+, \Rnm)} \\ \leq \left\|\varphi \partial^h_i \nabla \widetilde \bu\right\|_{L^2(U^+, \Rnm)} + 2 \left\| \nabla\varphi\right\|_{L^\infty(U^+, \Rm)}\left\| \nabla \widetilde \bu\right\|_{L^2(U^+, \Rnm)}.
\end{multline*} 
Hence, for every $\eps >0$ there exists $C_3(\eps)>0$ such that 
\begin{equation} \label{I2est} 
I_2 \leq \eps \left\|\varphi \partial^h_i \nabla \widetilde \bu\right\|_{L^2(U^+, \Rnm)}^2  + \left\| \nabla\varphi\right\|_{L^\infty(U^+, \Rm)}^2 \left\| \nabla \widetilde \bu\right\|_{L^2(U^+, \Rnm)}^2 +  C_3(\eps)\| \bff\|_{L^2(U^+, \Rn)}^2. 
\end{equation} 
Next, we estimate $I_3$ and $I_4$, for this purpose we  rewrite
\begin{multline}\label{pr6}
h \partial^h_i D \Phi (Q \nabla \widetilde \bu) \\
= D \Phi (Q(\cdot+ h\be_i) \nabla \widetilde \bu(\cdot+ h\be_i)) - D \Phi (Q \nabla \widetilde \bu(\cdot+ h\be_i))
+ D \Phi (Q \nabla \widetilde \bu(\cdot+ he_i)) - D \Phi (Q \nabla \widetilde \bu).
\end{multline}
With (\ref{pr6}) in mind, we estimate
\begin{multline*}
    \left| \partial^h_i (Q^T D \Phi (Q \nabla \widetilde \bu))\right| \leq |\partial^h_i Q^T| |D \Phi (Q \nabla \widetilde \bu)| + |Q^T( \cdot + h \be_i)| |\partial^h_i D \Phi(Q \nabla \widetilde \bu)| \\ \leq C_Q \mu (1 + |Q| |\nabla \widetilde \bu|) + C_Q \mu |\partial^h_i (Q \nabla \widetilde \bu)| \leq C_Q \mu(1 + C_Q |\nabla \widetilde \bu| + C_Q |\partial^h_i \nabla \widetilde \bu|) .
\end{multline*}
Thus,
\begin{multline} \label{I3est} 
|I_3| \leq \eps \left\|\varphi \partial^h_i \nabla \widetilde \bu\right\|_{L^2(U^+, \Rnm)}^2 \\+ C_4(\eps, \mu, C_Q)\left(|U^+| + \left(1 + \left\| \nabla\varphi\right\|_{L^\infty(U^+, \Rm)}^2\right)\left\| \nabla \widetilde \bu\right\|_{L^2(U^+, \Rnm)}^2\right). 
\end{multline} 
Estimating $I_4$ requires more care. Using the Leibniz rule, we obtain 
\begin{equation} 
I_4 = \int_{U^+} \varphi^2 (\partial^h_i Q^T) D \Phi(Q\, \nabla \widetilde\bu) \cddot  \nabla \partial^h_i \widetilde \bu 
+\int_{U^+} \varphi^2 \partial^h_i D \Phi(Q\, \nabla \widetilde\bu) \cddot Q \nabla \partial^h_i \widetilde \bu =: A_1 + A_2.
\end{equation}
Using (\ref{unifconv}, \ref{DPhi0}, \ref{CQ}), we estimate 
\begin{equation} 
|A_1| \leq \epsilon \|\varphi\partial^h_i \nabla \widetilde \bu \|_{L^2(U^+, \Rnm)}^2
+ C_5(\epsilon,C_Q,\mu) \| \nabla \widetilde \bu \|_{L^2(U^+, \Rnm)}^2.
\end{equation}
Recalling (\ref{pr6}), 
\begin{multline}
A_2 = \frac1h\int_{U^+}\varphi^2 (D \Phi(Q\, \nabla \widetilde\bu(\cdot+he_i)) 
- D \Phi(Q\, \nabla \widetilde\bu)))  \cddot Q(\partial^h_i  \nabla \widetilde \bu) 
\\+  \frac1h\int_{U^+}\varphi^2 (D \Phi(Q(\cdot+he_i)\, \nabla \widetilde\bu(\cdot+he_i)) 
- D \Phi(Q\, \nabla \widetilde\bu(\cdot+he_i))))  \cddot Q(\partial^h_i  \nabla \widetilde \bu) 
=: B_1 +B_2.
\end{multline}
We estimate $B_2$ similarly as $A_1$, 
\[|B_2| \leq
\epsilon \|\varphi\partial^h_i \nabla \widetilde \bu \|_{L^2(U^+, \Rnm)}^2
+ C_6(\epsilon,C_Q,\mu) \| \nabla \widetilde \bu \|_{L^2(U^+, \Rnm)}^2.\]
We deal differently with $B_1$. Using (\ref{unifconv}) and (\ref{CQ}) yields
\begin{equation}\label{pr7} B_1\geq \mu \|\varphi Q\partial^h_i \nabla \widetilde \bu \|_{L^2(U^+, \Rnm)}^2 \ge 
\frac{\mu}{C_Q} \|\varphi\partial^h_i \nabla \widetilde \bu \|_{L^2(U^+, \Rnm)}^2.
\end{equation} 
Collecting (\ref{diffquoteq1}-\ref{pr7}) and choosing $\eps$ small enough depending on $\mu, C_Q, \lambda$ we obtain 
\begin{multline*}\left\|\partial^h_i \nabla \widetilde \bu\right\|_{L^2(U_r^+, \Rnm)}^2 \leq \left\|\varphi \partial^h_i \nabla \widetilde \bu\right\|_{L^2(U^+, \Rnm)}^2 \\ \leq C_7\left(\mu, C_Q, \lambda, r, \left\| \nabla\varphi\right\|_{L^\infty(U^+, \Rm)}, \left\|\nabla \widetilde \bu\right\|_{L^2(U^+, \Rnm)}\right),
\end{multline*} 
whence $\widetilde \bu_{y_i} \in W^{1,2}(U_r^+, \Rnm)$ for $i = 1, \ldots, m-1$. 

In order to establish the missing estimate on $ \widetilde \bu_{y_my_m}$ it is advantageous to write (\ref{pr3}) as a differential equation,
\begin{equation} \label{ELeqAtilde} 
\widetilde \bu - \widetilde\bff = \lambda \div(Q^T D \Phi(Q\, \nabla \widetilde\bu)) \quad \text{in } U^+, \end{equation} 
\begin{equation} \label{ELbcAtilde} 
Q^T D \Phi(Q\, \nabla \widetilde\bu)\cdot \be_m = 0 \quad \text{on } U^0. 
\end{equation}
Expanding the divergence in \eqref{ELeqAtilde}, we obtain for $\alpha =1, \ldots, n$
\begin{multline} \label{expdiv} 
\widetilde u^\alpha - \widetilde f^\alpha = \sum_{i, j =1}^m Q_{ji, y_i} (D \Phi)^\alpha_j(Q \nabla \widetilde \bu) + \sum_{i, j,k,l =1}^m \sum_{\beta =1}^n Q_{ji}\, (D^2 \Phi)^{\alpha \beta}_{jk}(Q \nabla \widetilde \bu)\, Q_{kl, y_i}\, \widetilde u^\beta_{y_l} \\+ \sum_{i, j,k,l =1}^m \sum_{\beta =1}^n Q_{ji}\, (D^2 \Phi)^{\alpha \beta}_{jk}(Q \nabla \widetilde \bu)\, Q_{kl}\, \widetilde u^\beta_{y_l y_i}. 
\end{multline} 
We recall (\ref{unifconv}) to see that 
\[\left(\sum_{j,k =1}^m Q_{jm}\, (D^2 \Phi)^{\alpha \beta}_{jk}(Q \nabla \widetilde \bu)\, Q_{km}\right)_{\alpha, \beta = 1}^n \geq \frac{\mu}{C_Q^2} I^n.\]
Since we have already shown that $\widetilde \bu_{y_l y_i} \in L^2(U_r^+, \Rnm)$ as long as it is not the case that $l=i=m$, it follows from \eqref{expdiv} that also $\widetilde \bu_{y_m y_m} \in L^2(U_r^+, \Rnm)$. Thus, we have shown that $\widetilde \bu \in W^{2,2}(U_r^+, \Rnm)$ and therefore $\bu \in W^{2,2}(B_r(x_0)\cap \Omega, \Rnm)$. By compactness of $\partial \Omega$, it follows that $\bu \in W^{2,2}(\Omega, \Rnm)$. 
\end{proof}  

\noindent \textbf{Acknowledgement.} The authors thank an anonymous referee for their comments, which helped us to improve the text.

 \bibliographystyle{asdfgh}
 \bibliography{bib.bib.bib}
\end{document}